\documentclass[12pt, reqno]{amsart}

\usepackage{amssymb}
\usepackage{amsmath,amsfonts,amssymb}

\theoremstyle{plain}
\newtheorem{thm}{Theorem}

\newtheorem{eg}{Example}
\newtheorem{case}{Case}
\newtheorem{scase}{Sub-case}
\newtheorem{corollary}{Corollary}[section]

\theoremstyle{definition}

\theoremstyle{remark}
\newtheorem{remark}{Remark}

\numberwithin{equation}{section}
\numberwithin{lemma}{section}
\numberwithin{thm}{section}
\usepackage{amsmath}
\usepackage{amsfonts}   
\usepackage{amssymb}
\usepackage{amssymb, amsmath, amsthm}
\usepackage[breaklinks]{hyperref}

\usepackage{graphicx}
\usepackage{amsthm}

\begin{document} 

\title{ Generalized Fruit Diophantine equation and Hyperelliptic curves}
\author{Om Prakash and Kalyan Chakraborty}
\address{Kerala School of Mathematics, Kozhikode - 673571, Kerala, India.}
\email{omprakash@ksom.res.in}
\address{Kerala School of Mathematics, Kozhikode - 673571, Kerala, India.}
\email{kalychak@ksom.res.in}
\maketitle

\begin{abstract}
We show the insolvability of the Diophantine equation $ax^d-y^2-z^2+xyz-b=0$ in $\mathbb{Z}$ for fixed $a$ and $b$ such that $a\equiv 1 \pmod {12}$ and $b=2^da-3$, where $d$ is an odd integer and is a multiple of $3$. Further, we investigate the more general family with $b=2^da-3^r$, where $r$ is a positive odd integer. As a consequence, we found an infinite family of hyperelliptic curves with trivial torsion over $\mathbb{Q}$. We conclude by providing some numerical evidence corroborating the main results. 
\end{abstract}

\section{Introduction}
One of the earliest topics in number theory is the study of Diophantine equations. In the third century, Greek mathematician Diophantus of `Alexandria' began this study. A polynomial equation of the form 
$$
P(x_1,x_2,\cdots,x_n)=0
$$ 
is known as a Diophantine equation. Finding all of its integer solutions, or all of the $n-$tuples $\left(x_1,x_2,\cdots,x_n\right)\in \mathbb{Z}$ that satisfy the above equation, is of prime interest. The main task is to investigate whether solutions exist for a given Diophantine equation. If they do, it would be the aim to know how many are there and how to find all. There are certain Diophantine equations which has no non zero integer solutions, for example, Fermat's equation $x^n+y^n=z^n$ for $n\geq 3$. The tenth of Hilbert's 23 problems, which he presented in 1900, dealt with Diophantine equations. Hilbert asked, is there an algorithm to determine weather a given Diophantine equation has a solution or not? and Matiyasevich in 1970 answered it negatively.

We investigate a class of Diophantine equations of the form $ax^d-y^2-z^2+xyz-b=0$ for fixed $a$ and $b$. Due to its emergence when attempting to solve an equation involving fruits, this type of Diophantine equations were given the name ``Fruit Diophantine equation" by B. Sury and D. Majumdar \cite{sury} and they proved the following:
\begin{thm} \cite{sury} \label{s1}
The equation 
$$
y^2-xyz+z^2=x^3-5
$$ 
has no integer solution in $x$, $y$ and $z$.
\end{thm}

 Similar type of equations were previously studied by F. Luca and A. Togb\'e. In particular, Luca and Togb\'e \cite{LT} studied the solution of the Diophantine equation $x^3+by+1-xyz=0$ and later, Togb\'e \cite{T} independently studied the equation $x^3+by+4-xyz=0$. 
 
As a consequence of Theorem \ref{s1} Majumdar and Sury proved the following:
\begin{thm} \cite{sury}
For any integer $m$, the elliptic curve 
$$
E_m: y^2-mxy=x^3+m^2+5
$$ 
has no integral point. 
\end{thm}
L. Vaishya and R. Sharma expanded on Majumdar and Sury's work in \cite{lalit}. A class of fruit Diophantine equations without an integer solution was found by them. In particular Vaishya and Sharma showed,
\begin{thm} \cite{lalit} \label{l1}
For fixed integers $a$ and $b$ with $a\equiv 1 \pmod {12}$ and $b=8a-3$. The Diophantine equation $$ax^3-y^2-z^2+xyz-b=0$$ has no integer solution. 
\end{thm}
Using Nagell-Lutz theorem \cite{sil} and Theorem \ref{l1} they got hold of an infinite family of elliptic curves with torsion-free Mordell-Weil group over $\mathbb{Q}$. 
\begin{thm} \cite{lalit}
Let $a$ and $b$ be as in Theorem \ref{l1}. 
\begin{itemize}
\item For any even integer $m$ the elliptic curve 
$$
E_{m,a,b}^e: y^2=x^3+\frac{1}{4}m^2x^2-a^2\left(m^2+b\right)
$$ 
has torsion-free Mordell-Weil group. 
\item For any odd integer $m$ the elliptic curve $$E_{m,a,b}^o: y^2=x^3+m^2x^2-64a^2\left(m^2+b\right)$$ has torsion-free Mordell-Weil group. 
\end{itemize}
\end{thm}
We extend Vaishya and Sharma's results \cite{lalit} for higher exponents. We obtain a family of hyperelliptic curves, by carrying out some appropriate transformations. In 2013, D. Grant gave an analogue of Nagell-Lutz theorem for hyperelliptic curves \cite{david}, using which we conclude that the Mordell-Weil group of each member of the corresponding family of hyperelliptic curves is torsion-free. 
\section{Insolvability}
Here we state and prove the main theorem and derive a couple of interesting corollaries. We end this section by looking into a couple of examples.
\begin{thm} \label{t1}
The equation 
$$
ax^d-y^2-z^2+xyz-b=0
$$ 
has no integer solutions for fixed $a$ and $b$ such that $a\equiv 1 \pmod {12}$ and $b=2^da-3$, where $d$ is an odd integer and divisible by $3$.
\end{thm}
\begin{proof}
Consider 
\begin{eqnarray} \label{e1}
ax^d-y^2-z^2+xyz-b=0.
\end{eqnarray}
If possible, let $(x,y,z)$ be an integer solution of \eqref{e1}. 
Let us fix $x=\alpha$. 
Then \eqref{e1} can be re-written as,
\begin{eqnarray} \label{e2}
y^2+z^2+b=a{\alpha}^d+\alpha yz.
\end{eqnarray}
We consider the cases of $\alpha$ being even or odd separately.
\begin{case} \label{c1}
If $\alpha$ is even. Then, we write \eqref{e2} as:
\begin{eqnarray} \label{e23}
\left(y-\frac{\alpha z}{2}\right)^2 - \left(\frac{{\alpha}^2}{4}-1\right)z^2=a{\alpha}^d-b
\end{eqnarray}
and set $Y=y-\frac{\alpha z}{2}$, $\beta=\frac{\alpha}{2}$ and $z=Z$. Thus \eqref{e23} becomes, 
\begin{eqnarray} \label{e24}
Y^2-\left(\beta ^2 -1\right)Z^2=a{\alpha}^d-b=2^d\beta^d a-b.
\end{eqnarray}
\begin{itemize}
\item If $\beta$ is even, say $\beta = 2n$ for some integer $n$, then reducing \eqref{e24} modulo $4$ gives,
\begin{eqnarray}
Y^2+Z^2\equiv 3 \pmod 4,
\end{eqnarray}
which is not possible in $\mathbb{Z}/4\mathbb{Z}$.
\item If $\beta$ is odd, then $\beta=2n+1$ for some integer $n$. Reduction of \eqref{e24} modulo $4$ entails, 
\begin{eqnarray}
Y^2 \equiv 3 \pmod 4
\end{eqnarray}
which is impossible.
\end{itemize}
\end{case}
\begin{case}
If $\alpha$ is odd, say, $\alpha=2n+1$ for some integer $n$. Then,
\begin{eqnarray*}
y^2+z^2+b&=&a\alpha^d+\alpha yz\\
y^2+z^2+a2^d-3&=&a\left(2n+1\right)^d+\alpha yz\\
y^2+z^2-\left(2n+1\right)yz&=&a\left(2n+1\right)^d-a2^d+3.
\end{eqnarray*}
Now
\begin{eqnarray*}
y^2+z^2+yz&\equiv& a+3 \pmod 2,\\
\Rightarrow y^2+z^2+yz&\equiv& 0 \pmod 2.
\end{eqnarray*}
Note that $y^2+z^2+yz\equiv a+3 \pmod 2$ has only solution $y\equiv 0\equiv z$ in $\mathbb{Z}/2\mathbb{Z}$, that is, $y$ and $z$ are even. Thus \eqref{e23} becomes 
$$
a\alpha^d-b\equiv 0 \pmod 4.
$$
If we write $a=12l+1$ for some integer $l$, then, 
\begin{eqnarray*}
\alpha^d-\left(a2^d-3\right)&\equiv& 0 \pmod 4,\\
\Rightarrow \alpha^d+3 &\equiv& 0 \pmod 4,\\
\Rightarrow \alpha^d &\equiv& 1 \pmod 4,\\
\Rightarrow \alpha &\equiv& 1 \pmod 4.
\end{eqnarray*}
Let us consider 
\begin{eqnarray*} \label{e7}
\left(y-\frac{\alpha z}{2}\right)^2 - \left(\frac{{\alpha}^2}{4}-1\right)z^2&=&a{\alpha}^d-b,\\
\mbox{i.e.}~~\left(y-\frac{\alpha z}{2}\right)^2 - \left(\alpha^2-4\right)\left(\frac{z}{2}\right)^2&=&a{\alpha}^d-b.
\end{eqnarray*}
Further, we set $Y = y-\frac{\alpha z}{2}$ and $Z = \frac{z}{2}$. Then, 
\begin{eqnarray} \label{e7}
Y^2-\left(\alpha^2-4\right)Z^2=a{\alpha}^d-b
\end{eqnarray}
where $\alpha \equiv 1 \pmod 4$, $a\equiv 1 \pmod {12}$ and $b=a2^d-3$. Three sub cases need to be considered.
\begin{scase}
If $\alpha \equiv 1 \pmod {12}$, write $\alpha = 12l+1$ for some integer $l$. Then,
\begin{eqnarray*}
\alpha \equiv 1 \pmod 3\\
\Rightarrow \alpha+2\equiv 0 \pmod 3.
\end{eqnarray*}
Substituting $\alpha=12l+1$ in \ref{e7}, we get 
\begin{eqnarray*}
Y^2-\left(\left(12l+1\right)^2-4\right)Z^2&=&a{\alpha}^d-b,\\
\Rightarrow Y^2\equiv a\alpha^d-b \pmod 3,\\
\Rightarrow Y^2\equiv a\left(12l+1\right)^d-a2^d+3 \pmod 3,\\
\Rightarrow Y\equiv 1-2^d \pmod 3, \\
\Rightarrow Y^2\equiv 2 \pmod 3.
\end{eqnarray*}
A contradiction as $2$ is not square modulo $3$.
\end{scase}
\begin{scase}\label{sc2}
If $\alpha \equiv 9 \pmod {12}$. Then, there is a prime factor $p\equiv 5$  or $7 \pmod {12}$ of $(\alpha -2)$. Let $p\equiv 5$ or $7$ $\pmod {12}$ be a prime factor of $(\alpha -2)$. Thus,
\begin{eqnarray*}
Y^2\equiv a\alpha^d-b \pmod p.
\end{eqnarray*}
Let $\alpha = pl+2$ for some integer $l$. Then,
\begin{eqnarray*}
Y^2 &\equiv& a\left(pl+2\right)^d-b \pmod p,\\
\Rightarrow Y^2&\equiv& 3 \pmod p.
\end{eqnarray*}
This leads to a contradiction as $3$ is not a quadratic residue modulo $p$.
\end{scase}
\begin{scase} \label{sc3}
When $\alpha \equiv 5 \pmod {12}$, we substitute $\alpha =3k+2$ for some integer $k$ and get,
\begin{eqnarray*}
Y^2-\left(\left(3l+2\right)^2-4\right)Z^2&=&\left(12l+1\right)\left(3k+2\right)-2^d\left(12l+1\right)+3,\\
\Rightarrow Y^2&\equiv& 2-2^d \equiv 0 \pmod 3,\\
\Rightarrow Y&\equiv& 0 \pmod 3.
\end{eqnarray*}
Further, we substitute $Y=3m$ and $\alpha =12n+5$ for some integers $n$ and $m$ in \ref{e7} and arrive onto,
\begin{eqnarray*}
9m^2-\left(12n+3\right)\left(12n+7\right)Z^2&=&a\left(12n+5\right)^d-b,\\
\Rightarrow -\left(n+1\right) Z^2 &\equiv& \sum_{i=0}^{d-1}\left(12n+5\right)^{d-1-i}2^i \pmod 3,\\
\Rightarrow -\left(n+1\right) Z^2 &\equiv& 1 \pmod 3,\\
\Rightarrow n&\equiv& 1 \pmod 3.
\end{eqnarray*}
Hence, $\alpha \equiv 17 \pmod {36}$.

Note that $3$ divides $(\alpha-2)$. Thus there is a prime factor $p\equiv 5$ or $7 \pmod {12}$ of $\frac{(\alpha-2)}{3}$, otherwise it would mean that $\frac{\alpha-2}{3}$ is congruent to $\pm 1$, which is not the case. Therefore, 
\begin{eqnarray*}
\alpha-2\equiv 0 \pmod p.
\end{eqnarray*}
Thus,
\begin{eqnarray*}
Y^2\equiv a\alpha^d-b \pmod p.
\end{eqnarray*}
Substituting $\alpha=pl+2$ for some integer $l$, we have 
\begin{eqnarray*}
Y^2\equiv 3 \pmod p,
\end{eqnarray*}
which contradicts the fact that $3$ is quadratic residue modulo $p$ if $p\equiv \pm 1 \pmod {12}$.
\end{scase}
\end{case}
\end{proof}
\begin{remark}
The result of Sury and Majumdar \cite{sury} follows by substituting $a=1$ and $d=3$ in Theorem \ref{t1}. The particular case $d=3$ in the same theorem deduces the results of Vaishya and Sharma \cite{lalit}.
\end{remark}

By increasing the exponents in the expression for $b$ to 3, we will now examine the Diophantine equation with a little more generality. The potential of a solution in this scenario is described by the following two corollaries, along with a few examples.
\begin{corollary}\label{co1}
 The equation 
$$
ax^d-y^2-z^2+xyz-b=0
$$ 
has no integer solution $(x,y,z)$ with $x$ even for fixed integers $a$ and $b$ such that $a\equiv 1 \pmod {12}$ and $b=2^da-3^r$ with positive odd integers $r$ and $d$ as in Theorem \ref{t1}.
\end{corollary}
\begin{proof}
We follow exactly the same steps as in Case \ref{c1} of Theorem \ref{t1}. Suppose there is a solution with $x=\alpha$ even, then we write \eqref{e2} as:
\begin{eqnarray} \label{e3}
\left(y-\frac{\alpha z}{2}\right)^2 - \left(\frac{{\alpha}^2}{4}-1\right)z^2=a{\alpha}^d-b.
\end{eqnarray}
Let $Y= y-\frac{\alpha z}{2}, \beta=\frac{\alpha}{2}$ and $z=Z$. Then \eqref{e3} can be written as, 
\begin{eqnarray} \label{e4}
Y^2-\left(\beta ^2 -1\right)Z^2=a{\alpha}^d-b=2^d\beta^d a-b.
\end{eqnarray}
\begin{itemize}
\item If $\beta$ is even, say $\beta = 2n$ for some integer $n$, then the reduction modulo $4$ of \eqref{e4} will give,
\begin{eqnarray}
Y^2+Z^2\equiv 3^r \equiv 3 \pmod 4,
\end{eqnarray}
which is not feasible in $\mathbb{Z}/4\mathbb{Z}$.

\item If $\beta$ is odd, say $\beta=2n+1$ for some integer $n$. Then, the reduction modulo $4$ of \eqref{e4} provides, 
\begin{eqnarray}
Y^2 \equiv 3^r \equiv 3 \pmod 4,
\end{eqnarray}
which again is not possible.
\end{itemize}
\end{proof}
The following corollary deals with solutions having $x$, an odd integer:
\begin{corollary}\label{co2}
The equation 
$$
ax^d-y^2-z^2+xyz-b=0
$$ 
has no integer solution in $x$, $y$ and $z$ with $x\equiv 1$ or $9 \pmod {12}$,
for fixed integers $a, b$ such that $a\equiv 1 \pmod {12}$ and $b=2^da-3^r$, for $r$ and $d$ as in Corollary \ref{co1}.
\end{corollary}
\begin{proof}
Analogous steps as in Sub-case \ref{sc2} and \ref{sc3} of Theorem \ref{t1} will give the proof.
\end{proof}
\begin{remark} \label{r1}
Corollary \ref{co2} says that, if there is a solution of $ax^d-y^2-z^2+xyz-b=0$ with $a$ and $b$ as described in the Corollary \ref{co2}, then $x$ must be $5$ modulo $12$. 
\end{remark}
We will see some examples.
\begin{eg}\label{eg1}
For $a=25$, $d=3$ and $r=3$. The equation
\begin{eqnarray}\label{e212}
25x^3-y^2-z^2+xyz-173=0
\end{eqnarray}
has no integer solution.
\end{eg}
Example \ref{r1} shows that the equation may not have solution even with $x\equiv 5 \pmod {12}$. However, the next examples tell us the other possibility as well.
\begin{eg}\label{eg2}
If $a=13$, $d=3$ and $r=3$, then
\begin{eqnarray}\label{e212}
13x^3-y^2-z^2+xyz-77=0
\end{eqnarray}
has an integer solution $(5,=18,-102)$.
\end{eg}
\begin{remark} 
The condition that $r$ should be odd is rigid.
\end{remark}
\begin{eg} \label{eg3}
For $a=13$, $d=3$ and $r=2$, the equation
\begin{eqnarray}\label{e212}
13x^3-y^2-z^2+xyz-95=0
\end{eqnarray}
has an integer solution $(2,-10,-7)$. 
\end{eg}  
\section{Hyperelliptic curves}
A hyperelliptic curve $H$ over $\mathbb{Q}$ is a smooth projective curve associated to an affine plane curve given by the equation $y^2=f\left(x\right)$, where $f$ is a square-free polynomial of degree at least $5$. If the degree of $f$ is $2g+1$ or $2g+2$, then the curve has genus $g$. We write $H\left(\mathbb{Q}\right)$ for the set of $\mathbb{Q}$-points on $H$. Determining rational points on hyperelliptic curve is one of the major problems in mathematics. The following is the general result regarding the size of $H\left(\mathbb{Q}\right)$, which was conjectured by Mordell and was proved by Faltings:
\begin{thm} \cite{falt}
If $C$ is a smooth, projective and absolutely irreducible curve over $\mathbb{Q}$ of genus at least $2$, then $C\left(\mathbb{Q}\right)$ is finite.
\end{thm}
We may thus, at least theoretically, write down the finite set $C\left(\mathbb{Q}\right)$. It is still a significant unresolved problem to perform this practically for a given curve. 

Given a hyperelliptic curve $H$, we can define the {\it height} (classical) function to be the maximum of absolute values of the coefficients. The Northcott property tells us that there are finitely many equations with bounded height. Thus, one may talk about the density and averages. In this regard, Bhargava \cite{manjul} has proved that most of the hyperelliptic curve over $\mathbb{Q}$ has no rational point. So, most of the times calculating $H\left(\mathbb{Q}\right)$ means proving $H\left(\mathbb{Q}\right)=\phi$. 

In this section, we construct hyperelliptic curves corresponding to the  equation $ax^d-y^2-z^2+xyz-b=0$ with $a$ and $b$ as mentioned in Theorem \ref{t1}. Then, we prove that $H\left(\mathbb{Q}\right)=\phi$ (corroborating Bhargava \cite{manjul}). The main ingredient to prove this is the following Nagell-Lutz type theorem (Theorem 3, \cite{david}) proved by D. Grant.

\begin{thm}\cite{david}
    Let $C$ be a nonsingular projective curve of genus $g\geq 1$ given by $y^2=x^{2g+1}+b_1x^{2g}+\cdots+b_{2g}x+b_{2g+1}$, where $b_i\in \mathbb{Z}$. Suppose $$\psi : C\left(\mathbb{Q}\right)\rightarrow J\left(\mathbb{Q}\right)$$ be the Abel-Jacobi map, defined by $\psi\left(p\right)=\left[p-\infty\right]$, where $J\left(\mathbb{Q}\right)$ is the Jacobian variety. If $p=\left(x,y\right)\in C\left(\mathbb{Q}\right)\setminus \{\infty\}$ and $\psi \left(p\right)\in J\left(\mathbb{Q}\right)_{\text{tors}}$, then, $x,y\in \mathbb{Z}$ and either $y=0$ or $y^2$ divides discriminant of the polynomial $x^{2g+1}+b_1x^{2g}+\cdots+b_{2g}x+b_{2g+1}$.
\end{thm}

 For fixed $m$ we define hyperelliptic curves, 
\begin{eqnarray*}
H_{m,a,b}: y^2-mxy = ax^d-m^2-b.
\end{eqnarray*}
\begin{itemize}
\item Suppose $m$ is even. Then write \eqref{e1} as:
\begin{eqnarray} \label{e31}
\left(y-\frac{mx}{2}\right)^2-\frac{m^2x^2}{4}=ax^d-m^2-b.
\end{eqnarray}
Multiplying \eqref{e31} by $a^{d-1}$ throughout, and using the fact that $d$ is odd and divisible by $3$, we have,
\begin{eqnarray} \label{e32}
\left(\left(y-\frac{mx}{2}\right)a^{\frac{d-1}{2}}\right)^2-a^{d-1}\frac{m^2x^2}{4}=\left(ax\right)^d-m^2a^{d-1}-ba^{d-1}.
\end{eqnarray}
We get the following hyperelliptic curve by substituting $\left(\left(y-\frac{mx}{2}\right)a^{\frac{d-1}{2}}\right)=Y$ and $ax=X$,  
\begin{eqnarray} \label{ee}
H_{m,a,b}^e:  Y^2-a^{d-3}\frac{m^2X^2}{4}=X^d-m^2a^{d-1}-ba^{d-1}.
\end{eqnarray}
\item Now if  $m$ is odd, multiply \eqref{e32} by $4^d$ throughout to get
\begin{eqnarray*}
\left(\left(y-\frac{mx}{2}\right)a^{\frac{d-1}{2}}2^d\right)^2-\left(4a\right)^{d-1}{m^2x^2}=\left(4ax\right)^d-m^2a^{d-1}4^d-ba^{d-1}4^d.
\end{eqnarray*}
Finally substitute $ \left(\left(y-\frac{mx}{2}\right)a^{\frac{d-1}{2}}2^d\right)=Y$ and $4ax=X$, to get 
\begin{eqnarray} \label{eo}
H_{m,a,b}^o:  Y^2-\left(4a\right)^{d-3}{m^2X^2}=X^d-m^2a^{d-1}4^d-ba^{d-1}4^d.
\end{eqnarray}
\end{itemize}
Let,
\begin{eqnarray}
H_{m,a,b} = \begin{cases}
H_{m,a,b}^e & {\rm ~if~} m\hspace{0.1cm} \text{is even} \\
H_{m,a,b}^o & {\rm ~if~} m\hspace{0.1cm} \text{is odd},
\end{cases}
\end{eqnarray}
be the hyperelliptic curves.
\begin{thm} \label{t2}
Let $a$ and $b$ be as defined in Theorem \ref{t1}. For any $m\in \mathbb{N}$, the hyperelliptic curve $H_{m,a,b}$ has torsion-free Mordell-Weil group over $\mathbb{Q}$.
\end{thm}
\begin{proof}
Let $a$ and $b$ be fixed positive integers with $a\equiv 1 \pmod {12}$ and $b=2^da-3$. 
\begin{itemize}
\item For any even integer $m$, consider the hyperelliptic curve 
\begin{equation} \label{e36}
H_{m,a,b}^e:  Y^2-a^{d-3}\frac{m^2X^2}{4}=X^d-m^2a^{d-1}-ba^{d-1}.
\end{equation}
By Theorem $3$ of \cite{david}, if \eqref{e36} has an integer solution $\left(X_0, Y_0\right)$,\\ then $\left(aX_0, \left(\left(Y_0-\frac{mX_0}{2}\right)a^{\frac{d-1}{2}}\right), m\right)$ is a solution of  \eqref{e1}. However, in Theorem \ref{t1} we have proved that it has no integer solutions.

\item For an odd integer $m$, consider the hyperelliptic curve 
\begin{equation} \label{e37}
H_{m,a,b}^o:  Y^2-\left(4a\right)^{d-3}{m^2X^2}=X^d-m^2a^{d-1}4^d-ba^{d-1}4^d.
\end{equation}
Suppose \eqref{e37} has a solution $\left(X_0, Y_0\right)$, then $\left(4aX_0, \left(\left(Y_0-\frac{mX_0}{2}\right) a^{\frac{d-1}{2}}2^d\right), m\right)$ is a solution of \eqref{e1}, which is a contradiction.
\end{itemize}
\end{proof}

\section{Numerical examples}
In this section we give some numerical examples corroborating our results in Corollary \ref{co2} and Remark \ref{r1}.
\begin{center}
\begin{tabular}{||c c c c c||} 
 \hline
 $a$ & $d$ & $r$ & Equation & Solution\\ [1ex] 
 \hline\hline
1 & 3 & 3 & $x^3-y^2-z^2+xyz+19=0$ & $\left(5,0,-12\right)$\\
 \hline
 1 & 3 & 5 & $x^3-y^2-z^2+xyz+235=0$ & $\left(29,12,-60\right)$\\
 \hline
 1 & 3 & 7 & $x^3-y^2-z^2+xyz+2179=0$ & $\left(5,0,-48\right)$\\
 \hline
 1 & 3 & 9 & $x^3-y^2-z^2+xyz+19675=0$ & $\left(-31,12,-30\right)$\\
 \hline
 13 & 3 & 3 & $13x^3-y^2-z^2+xyz-77=0$ & $\left(5,-18,-102\right)$\\ 
 \hline
 13 & 3 & 5 & $13x^3-y^2-z^2+xyz+139=0$ & $\left(5,0,-42\right)$\\
 \hline
 13 & 3 & 7 & $13x^3-y^2-z^2+xyz+2083=0$ & ?\\
 \hline
 25 & 3 & 3 & $25x^3-y^2-z^2+xyz-173=0$ & $\left(5,0,-42\right)$\\ [1ex]
 
 \hline
\end{tabular}
\end{center}
\section*{Acknowledgement} This work is done during the  first author's  visit to Institute of Mathematical Sciences (IMSc), Chennai, and he is grateful to the Institute for the hospitality and the wonderful working ambience. Both the authors are grateful to Kerala School of Mathematics(KSoM), Kozhikode, for it's support and wonderful ambience.

\end{document}